\documentclass{amsart}
\usepackage[all]{xy}
\usepackage{verbatim}
\usepackage{color}
\usepackage{amsthm}
\usepackage{amssymb}
\usepackage[colorlinks=true]{hyperref}
\usepackage{footmisc}

\numberwithin{equation}{section}

\newtheorem{theorem}[equation]{Theorem}
\newtheorem*{theorem*}{Theorem} \newtheorem{lemma}[equation]{Lemma}

\newtheorem*{conjecture*}{Mamma Conjecture}
\newtheorem*{conjecture1*}{Mamma Conjecture (revisited)}
\newtheorem{proposition}[equation]{Proposition}
\newtheorem{corollary}[equation]{Corollary}
\newtheorem*{corollary*}{Corollary}

\theoremstyle{remark}
\newtheorem{definition}[equation]{Definition}

\newtheorem{example}[equation]{Example}

\newtheorem{notation}[equation]{Notation}

\theoremstyle{remark}
\newtheorem{remark}[equation]{Remark}

\newcommand{\cA}{{\mathcal A}}
\newcommand{\cB}{{\mathcal B}}
\newcommand{\cC}{{\mathcal C}}
\newcommand{\cD}{{\mathcal D}}

\newcommand{\cT}{{\mathcal T}}

\newcommand{\cX}{{\mathcal X}}

\newcommand{\Spt}{\mathrm{Spt}}

\newcommand{\bbA}{\mathbb{A}}

\newcommand{\bbC}{\mathbb{C}}

\newcommand{\bbN}{\mathbb{N}}

\newcommand{\bbS}{\mathbb{S}}

\newcommand{\bbZ}{\mathbb{Z}}

\DeclareMathOperator{\id}{id}

\DeclareMathOperator{\Mot}{Mot}

\newcommand{\dgcat}{\mathrm{dgcat}} 

\newcommand{\loc}{\mathrm{loc}}

\newcommand{\bbK}{I\mspace{-6.mu}K}

\newcommand{\perf}{\mathrm{perf}}

\newcommand{\dg}{\mathrm{dg}}

\newcommand{\Hom}{\mathrm{Hom}}

\newcommand{\rep}{\mathrm{rep}}

\newcommand{\Hmo}{\mathrm{Hmo}}
\newcommand{\op}{\mathrm{op}}

\newcommand{\too}{\longrightarrow}

\newcommand{\add}{\mathrm{add}}

\newcommand{\ie}{\textsl{i.e.}\ }

\let\oldmarginpar\marginpar
\def\marginpar#1{\oldmarginpar{\tiny #1}}

\begin{document}

\title[Invariants of corner skew Laurent polynomial algebras]{$\bbA^1$-homotopy invariants of \\corner skew Laurent polynomial algebras}
\author{Gon{\c c}alo~Tabuada}

\address{Gon{\c c}alo Tabuada, Department of Mathematics, MIT, Cambridge, MA 02139, USA}
\email{tabuada@math.mit.edu}
\urladdr{http://math.mit.edu/~tabuada}
\thanks{The author was partially supported by a NSF CAREER Award}

\date{\today}

\abstract{In this note we prove some structural properties of all the $\bbA^1$-homotopy invariants of corner skew Laurent polynomial algebras. As an application, we compute the mod-$l$ algebraic $K$-theory of Leavitt path algebras using solely the kernel/cokernel of the incidence matrix. This leads naturally to some vanishing and divisibility properties of the $K$-theory of these algebras.}
}

\maketitle
\vskip-\baselineskip
\vskip-\baselineskip


\section{Corner skew Laurent polynomial algebras}
Let $k$ be a field, $A$ a unital $k$-algebra, $e$ an idempotent of $A$, and $\phi\colon A\stackrel{\sim}{\to}~eAe$ a ``corner'' isomorphism. Following Ara-Barroso-Goodearl-Pardo \cite[\S2]{Fractional}, the associated {\em corner skew Laurent polynomial algebra $A[t_+,t_-;\phi]$} is defined as follows: the elements are formal expressions $t^m_- a_{-m} + \cdots + t_- a_{-1} + a_0 + a_1 t_+ \cdots + a_n t_+^n$
with $a_{-i} \in \phi^i(1)A$ and $a_i \in A\phi^i(1)$ for every $i \geq 0$; the addition is defined componentwise; the multiplication is determined by the distributive law and by the relations $t_-t_+=1$, $t_+t_-=e$, $at_-=t_-\phi(a)$ for every $a \in A$, and $t_+a = \phi(a)t_+$ for every $a \in A$. Note that $A[t_+,t_-;\phi]$ admits a canonical $\bbZ$-grading with $\mathrm{deg}(t_\pm)=\pm 1$. 

As proved in \cite[Lem.~2.4]{Fractional}, the corner skew Laurent polynomial algebras can be characterized as those $\bbZ$-graded algebras $C=\bigoplus_{n \in \bbZ}C_n$ containing elements $t_+\in C_1$ and $t_-\in C_{-1}$ such that $t_-t_+=1$. Concretely, we have $C=A[t_+,t_-;\phi]$ with $A:=C_0$, $e:=t_+t_-$, and $\phi\colon C_0 \to t_+ t_- C_0 t_+ t_-$ given by $c_0\mapsto t_+ c_0 t_-$.
\begin{example}[Skew Laurent polynomial algebras]
When $e=1$, $A[t_+,t_-;\phi]$ reduces to the classical skew Laurent polynomial algebra $A \rtimes_\phi \bbZ$. In the particular case where $\phi$ is the identity, $A\rtimes_\phi \bbZ$ reduces furthermore to $A[t,t^{-1}]$.
\end{example}
\begin{example}[Jacobson algebras]
Following \cite{Jacobson}, the {\em Jacobson algebra $J_n$, $n\geq 0$,} is the $k$-algebra generated by elements $x_0, \ldots, x_n, y_0, \ldots, y_n$ subject to the relations $y_i x_j = \delta_{ij}$. Note that the canonical $\bbZ$-grading, with $\mathrm{deg}(x_i)=1$ and $\mathrm{deg}(y_i)=-1$, makes $J_n$ into a corner skew Laurent polynomial algebra. The algebras $J_n$ are also usually called Cohn algebras (see \cite{AK}), and $J_0$ the (algebraic) Toeplitz algebra.
\end{example}
\begin{example}[Leavitt algebras]
Following \cite{Leavitt}, the {\em Leavitt algebra $L_n$, $n\geq 0$,} is the $k$-algebra generated by elements $x_0, \ldots, x_n, y_0, \ldots, y_n$ subject to the relations $y_i x_j = \delta_{ij}$ and $\sum^n_{i=0} x_i y_i =1$. Note the canonical $\bbZ$-grading, with $\mathrm{deg}(x_i)=1$ and $\mathrm{deg}(y_i)=-1$, makes $L_n$ into a corner skew Laurent polynomial algebra. Note also that  $L_0\simeq k[t,t^{-1}]$. In the remaining case $n\geq 1$, $L_n$ is the universal example of a $k$-algebra of {\em module type $(1,n+1)$}, \ie $L_n\simeq L_n^{\oplus (n+1)}$ as right $L_n$-modules.
\end{example}
\begin{example}[Leavitt path algebras]\label{ex:Leavitt-path}
Let $Q=(Q_0,Q_1, s,r)$ be a finite quiver; $Q_0$ and $Q_1$ stand for the sets of vertex and arrows, respectively, and $s$ and $r$ for the source and target maps, respectively. We assume that $Q$ has no {\em sources}, \ie vertices $i\in Q_0$ such that $\{\alpha\,|\, r(\alpha)=i\} =\emptyset$. Consider the double quiver $\overline{Q}=(Q_0,Q_1\cup Q_1^\ast, s,r)$ obtained from $Q$ by adding an arrow $\alpha^\ast$, in the converse direction, for each arrow $\alpha \in Q_1$. Following Abrams-Pino \cite{AP} and Ara-Moreno-Pardo \cite{Pardo}, the {\em Leavitt path algebra $L_Q$ of $Q$} is the quotient of the quiver algebra $k\overline{Q}$ (which is generated by elements $\alpha \in Q_1\cup Q_1^\ast$ and $e_i$ with $i \in Q_0$) by the Cuntz-Krieger's relations: $\alpha^\ast \beta = \delta_{\alpha \beta} e_{r(\alpha)}$ for every $\alpha, \beta \in Q_1$; $\sum_{\{\alpha \in Q_1|s(\alpha)=i\}} \alpha \alpha^\ast =e_i$ for every non-sink $i \in Q_0$.
Note that $L_Q$ admits a canonical $\bbZ$-grading with $\mathrm{deg}(\alpha)=1$ and $\mathrm{deg}(\alpha^\ast)=-1$. For every vertex $i \in Q_0$ choose an arrow $\alpha_i$ such that $r(\alpha_i)=i$ and consider the associated elements $t_+:=\sum_{i \in Q_0}\alpha_i$ and $t_-:=t_+^\ast$. Since $\mathrm{deg}(t_\pm)=\pm1$ and $t_-t_+=1$, $L_Q$ is also an example of a corner skew Laurent polynomial algebra. 

In the particular case where $Q$ is the quiver with one vertex and $n+1$ arrows,~$L_Q$ is isomorphic to $L_n$. Similarly, when $Q$ is the quiver with two vertices $\{1,2\}$ and $2(n+1)$ arrows ($n+1$ from $1$ to $1$ and $n+1$ from $1$ to $2$), we have $L_Q\simeq J_n$.
\end{example}
\section{$\bbA^1$-homotopy invariants}
A {\em dg category $\cA$}, over a base field $k$, is a category enriched over cochain complexes of $k$-vector spaces; see \S\ref{sub:dg}. Every (dg) $k$-algebra $A$ gives rise to a dg category with a single object. Another source of examples is provided by schemes since the category of perfect complexes $\perf(X)$ of every quasi-compact quasi-separated $k$-scheme $X$ admits a canonical dg enhancement $\perf_\dg(X)$. In what follows, we denote by $\dgcat(k)$ the category of (small) dg categories and dg functors.

A functor $E\colon \dgcat(k) \to \cT$, with values in a triangulated category, is called a {\em localizing invariant} if it satisfies the following three conditions:
\begin{itemize}
\item it inverts the derived Morita equivalences (see \S\ref{sub:dg});
\item it sends\footnote{\label{note1}In a way which is functorial for strict morphisms of sequential colimits and strict morphisms of short exact sequences of dg categories.} sequential (homotopy) colimits to sequential homotopy colimits;
\item it sends\footref{note1} short exact sequences of dg categories, in the sense of Drinfeld \cite{Drinfeld} and Keller \cite{Keller} (see \cite[\S4.6]{ICM-Keller}), to distinguished triangles
\begin{eqnarray*}
0 \too \cA \too \cB \too \cC \too 0 &\mapsto & E(\cA) \too E(\cB) \too E(\cC) \stackrel{\partial}{\too} \Sigma E(\cA)\,.
\end{eqnarray*}
\end{itemize}
When $E$ inverts moreover the dg functors $\cA \to \cA[t]$, where $\cA[t]$ stands for the tensor product $\cA\otimes k[t]$, we call it an {\em $\bbA^1$-homotopy invariant}. Examples of localizing invariants include Hochschild homology $HH$, topological Hochschild homology $THH$, cyclic homology $HC$, the mixed complex $C$, nonconnective algebraic $K$-theory $\bbK$, mod-$l^\nu$ nonconnective algebraic $K$-theory $\bbK(-;\bbZ/l^\nu)$ (with $l^\nu$ a prime power) and its variant $\bbK(-)\otimes \bbZ[1/l]$, homotopy $K$-theory $KH$, mod-$l^\nu$ homotopy $K$-theory $KH(-;\bbZ/l^\nu)$, and \'etale $K$-theory $K^{\mathrm{et}}(-;\bbZ/l^\nu)$; see \cite[\S8.2]{book}. Among those, $\bbK(-;\bbZ/l^\nu)$ (with $l \nmid \mathrm{char}(k)$), $\bbK(-)\otimes \bbZ[1/l]$ (with $l=\mathrm{char}(k)$), $KH$, $KH(-;\bbZ/l^\nu)$, and $K^{\mathrm{et}}(-;\bbZ/l^\nu)$ are $\bbA^1$-homotopy invariants; see \cite[\S8.5]{book}. When applied to $A$, resp. to $\perf_\dg(X)$, the preceding invariants reduce to the corresponding invariants of the (dg) $k$-algebra $A$, resp. of the $k$-scheme~$X$.
\begin{example}[Noncommutative mixed motives]\label{ex:Mot}
  Let $\Mot_\loc(k)$, resp. $\Mot_\loc^{\bbA^1}(k)$, be the category of noncommutative mixed motives constructed in \cite[\S 8.2]{book}, resp. in
  \cite[\S 8.5.1]{book}. By
  construction, this triangulated category comes equipped with a localizing invariant $\mathrm{U}_\loc\colon \dgcat(k) \to \Mot_\loc(k)$, resp. with an $\bbA^1$-homotopy invariant $\mathrm{U}_\loc^{\bbA^1}\colon \dgcat(k) \to \Mot_\loc^{\bbA^1}(k)$, which is initial among all localizing invariants, resp. among all $\bbA^1$-homotopy invariants; consult \cite[\S8-9]{book} for further details.
 \end{example}
\section{Statement of results and applications}\label{sec:results}
Let $A$ be a unital $k$-algebra, $e$ an idempotent of $A$, and $\phi\colon A \stackrel{\sim}{\to} eAe$ a ``corner'' isomorphism. Out of this data, we can construct the associated corner skew Laurent polynomial algebra $A[t_+,t_-;\phi]$. In the same vein, let us write $A[t_+;\phi]$ for the algebra defined by the formal expressions $a_0 + a_1 t_+ + a_2 t^2_+ + \cdots + a_n t^n_+$ and by the relations $t_+a=\phi(a)t_+$. Consider the $A\text{-}A$-bimodule ${}_\phi A$ associated to $\phi$ (see Notation \ref{not:unital}) and the {\em square-zero extension} $A[\epsilon]:=A\ltimes \!({}_\phi A[1])$ of $A$ by the suspension ${}_\phi A[1]$ of ${}_\phi A$. Concretely, $A[\epsilon]= A \oplus (eA)[1]$ with multiplication law $(a, eb)\cdot (a', eb'):= (aa', e(ab'+ba'))$. Consider also the dg $A[\epsilon]\text{-}A[t_+;\phi]$-bimodule $\mathrm{B}$ whose restriction to $A[t_+;\phi]$ is the projective resolution $t_+\cdot - \colon A[t_+;\phi] \to eA[t_+;\phi]$ of the trivial right $A[t_+;\phi]$-module $eA$. The left dg action is given by the canonical identification of $A[\epsilon]$ with the Ext-algebra $Ext^\ast_{A[t_+;\phi]}(eA)$. As explained in \S\ref{sub:dg}, the dg bimodule $\mathrm{B}$ corresponds to a morphism from $A[\epsilon]$ to $A[t_+;\phi]$ in the localization of $\dgcat(k)$ with respect to derived Morita equivalences. Therefore, given any functor $E$ which inverts derived Morita equivalences, we obtain an induced morphism $E(\mathrm{B})\colon E(A[\epsilon]) \to E(A[t_+;\phi])$.
\begin{theorem}\label{thm:main1}
For every localizing invariant $E$, we have the homotopy colimit
\begin{equation}\label{eq:sequence}
\mathrm{hocofib}\, E(\mathrm{B}) \too \mathrm{hocofib}\,E(\mathrm{B}) \too \cdots \too E(A[t_+,t_-;\phi])\,,
\end{equation}
where the transition morphism(s) is induced by the corner isomorphism $\phi$.

When $E$ is an $\bbA^1$-homotopy invariant, $E(\mathrm{B})$ reduces to $\id- E({}_\phi A)\colon E(A) \to~\!\!E(A)$ and the composition \eqref{eq:sequence} is an isomorphism. Consequently, we obtain a triangle
\begin{equation}\label{eq:triangle-searched}
E(A) \stackrel{\id- E({}_\phi A)}{\too} E(A) \too E(A[t_+,t_-;\phi])\stackrel{\partial}{\too} \Sigma E(A)\,.
\end{equation}
\end{theorem}
\begin{remark}[Generalization]
Given a dg category $\cA$, Theorem \ref{thm:main1} holds more generally with $A$ and $A[t_+,t_-;\phi]$ replaced by $\cA\otimes A$ and $\cA\otimes A[t_+,t_-;\phi]$; see \S\ref{sub:generalization}.
\end{remark}
\begin{corollary}\label{cor:main1}
Given a dg category $\cA$, we have a distinguished triangle of spectra:
$$ KH(\cA\otimes A) \stackrel{\id- KH({}_{\id}\cA \otimes {}_\phi A)}{\too} KH(\cA\otimes A) \to KH(\cA\otimes A[t_+,t_-;\phi])\stackrel{\partial}{\to} \Sigma KH(\cA\otimes A)\,.
$$
\end{corollary}
\begin{remark}[Related work]\label{rk:work}
Given a (not necessarily unital) $k$-algebra $C$, Ara-Brustenga-Corti\~nas proved in \cite[Thms.~3.6 and 8.4]{Munster} the analogue of Corollary \ref{cor:main1} with $\cA$ replaced by $C$. Their proof, which is inspired from operator $K$-theory, makes essential use of non-unital algebras. Since these latter objects don't belong to the realm of dg categories, our proof of Corollary \ref{cor:main1} is necessarily different. Among other ideas used in the proof, we show that every corner skew Laurent polynomial algebra is derived Morita equivalent to a dg orbit category in the sense of Keller; see Proposition \ref{prop:new}. Finally, note that in contrast with Ara-Brustenga-Corti\~nas' result, Corollary \ref{cor:main1} (and more generally Theorem \ref{thm:main1}) holds for dg algebras~and~schemes.
\end{remark}
Let $Q=(Q_0,Q_1,s,r)$ be a finite quiver, without sources, with $v$ vertices and $v'$ sinks. We assume that the set $Q_0$ is ordered with the first $v'$ elements corresponding to the sinks. Let $I'_Q$ be the incidence matrix of $Q$, $I_Q$ the matrix obtained from $I'_Q$ by removing the first $v'$ rows (which are zero), and $I_Q^t$ the transpose of $I_Q$.
\begin{theorem}\label{thm:main2}
Let $\cA$ be a dg category and $Q$ a finite quiver without sources. For every $\bbA^1$-homotopy invariant $E$, we have a distinguished triangle
$$ \oplus_{i=1}^{v-v'} E(\cA) \stackrel{\binom{0}{\id} - I^t_Q}{\too} \oplus_{i=1}^v E(\cA) \too E(\cA\otimes L_Q) \stackrel{\partial}{\too} \oplus^{v-v'}_{i=1} \Sigma E(\cA)\,.
$$
\end{theorem}
Theorem \ref{thm:main2} shows that all the information concerning $\bbA^1$-homotopy invariants of Leavitt path algebras $L_Q$ is encoded in the incidence matrix of the quiver $Q$.
\begin{example}[Jacobson algebras]
Let $Q$ be the quiver with two vertices $\{1,2\}$ and $2(n+1)$ arrows ($n+1$ from $1$ to $1$ and $n+1$ from $1$ to $2$). In this particular case, the distinguished triangle of Theorem \ref{thm:main2} reduces to
$$ E(\cA) \stackrel{\binom{n+1}{n}}{\too} E(\cA) \oplus E(\cA) \too E(\cA\otimes J_n) \stackrel{\partial}{\too} \Sigma E(\cA)\,.$$
Since $(n,n+1)=1$, we conclude that $E(\cA\otimes J_n) \simeq E(\cA)$. This shows that the $\bbA^1$-homotopy invariants don't distinguish the Jacobson algebras $J_n$ from $k$. Note that $J_n$ is much bigger than $k$; for instance, it contains the path algebra $kQ$.
\end{example}
\begin{example}[Leavitt algebras]\label{ex:Leavitt}
Let $Q$ be the quiver with one vertex and $n+1$ arrows. In this particular case, the distinguished triangle of Theorem \ref{thm:main2}~reduces~to
\begin{equation}\label{eq:triangle}
E(\cA) \stackrel{n\cdot \id}{\too} E(\cA) \too E(\cA\otimes L_n) \stackrel{\partial}{\too} \Sigma E(\cA)\,.
\end{equation}
When $n=0$, the distinguished triangle \eqref{eq:triangle} splits and gives rise to the ``fundamental'' isomorphism $E(\cA\otimes L_0)\simeq E(\cA) \oplus \Sigma E(\cA)$. When $n=1$, we have $E(\cA\otimes L_1)=0$. In the remaining case $n \geq 2$, $E(\cA\otimes L_n)$ identifies with the mod-$n$ Moore object of $E(\cA)$. Intuitively speaking, this shows that the functor $\cA \mapsto \cA\otimes L_n$, with $n \geq 2$, is a model of the mod-$n$ Moore construction.
\end{example}
\begin{proposition}\label{prop:very-last}
Let $l_1^{\nu_1} \times \cdots \times l_r^{\nu_r}$ be the prime decomposition of an integer $n \geq 2$. For every dg category $\cA$ and $\bbA^1$-homotopy invariant $E$, we have a direct sum decomposition $E(\cA\otimes L_n) \simeq E(\cA \otimes L_{l_1^{\nu_1}})\oplus \cdots \oplus E(\cA \otimes L_{l_r^{\nu_r}})$.
\end{proposition}
Roughly speaking, Proposition \ref{prop:very-last} shows that all the $\bbA^1$-homotopy invariants of Leavitt algebras are ``$l$-local''. Note that $L_n \not \simeq L_{l_1^{\nu_1}}\times \cdots \times L_{l_r^{\nu_r}}$.
\begin{remark}[Homotopy $K$-theory]\label{rk:KH}
By taking $E=KH$ in Theorem \ref{thm:main2}, we obtain a distinguished triangle of spectra
\begin{equation*}
\oplus^{v-v'}_{i=1} KH(\cA) \stackrel{\binom{0}{\id} - I^t_Q}{\too} \oplus^v_{i=1} KH(\cA) \to KH(\cA\otimes L_Q) \stackrel{\partial}{\to} \oplus^{v-v'}_{i=1} \Sigma KH(\cA)\,.
\end{equation*}
Given a (not necessarily unital) $k$-algebra $C$, Ara-Brustenga-Corti\~nas constructed in \cite[Thm.~8.6]{Munster} the analogue of the preceding distinguished triangle with $\cA$ replaced by $C$. Our construction is different and applies also to dg categories and schemes.
\end{remark}
\begin{remark}\label{rk:very-last}
By taking $E=KH$ and $n=l^\nu$ in Example \ref{ex:Leavitt}, we obtain an isomorphism between $KH(\cA\otimes L_{l^\nu})$ and the mod-$l^\nu$ homotopy $K$-theory spectrum $KH(\cA;\bbZ/l^\nu)$. When $l \nmid \mathrm{char}(k)$, the latter spectrum is isomorphic to $\bbK(\cA;\bbZ/l^\nu)$.
\end{remark}
\subsection*{Mod-$l^\nu$ algebraic $K$-theory of Leavitt path algebras}
Let $l^\nu$ be a prime power such that $l \neq \mathrm{char}(k)$ and $Q$ a finite quiver without sources. By taking $\cA=k$ and $E=\bbK(-;\bbZ/l^\nu)$ in Theorem \ref{thm:main2}, we obtain a distinguished triangle of spectra
$$ \oplus^{v-v'}_{i=1} \bbK(k;\bbZ/l^\nu) \stackrel{\binom{0}{\id} - I^t_Q}{\too} \oplus^v_{i=1} \bbK(k;\bbZ/l^\nu) \to \bbK(L_Q; \bbZ/l^\nu) \stackrel{\partial}{\to} \oplus^{v-v'}_{i=1} \Sigma \bbK(k;\bbZ/l^\nu)\,.$$
\begin{remark}
The preceding triangle follows also from the work of Ara-Brustenga-Corti\~nas \cite{Munster}. Indeed, since by hypothesis $l \nmid \mathrm{char}(k)$, the functors $\bbK(-;\bbZ/l^\nu)$ and $KH(-;\bbZ/l^\nu)$ are isomorphic. Moreover, as explained in Remark \ref{rk:very-last}, the latter functor identifies with $KH(-\otimes L_{l^\nu})$. Therefore, if in Remark \ref{rk:KH} we take for $C$ the $k$-algebra $L_{l^\nu}$, we obtain the preceding distinguished triangle of spectra.
\end{remark}
Assume that $k$ is algebraically closed. As proved by Suslin\footnote{Given a quiver $Q$, let $C^\ast_\bbC(Q)$ be the associated Cuntz-Krieger $C^\ast$-algebra. Corti\~nas kindly informed the author that the work of Suslin was also used in \cite[Thm.~9.4]{Munster} in order to prove that $\bbK_n(\bbC \otimes L_Q) \simeq K_n^{\mathrm{top}}(C_\bbC^\ast(Q)), n \geq 0$, for every quiver $Q$ without sinks such that $\mathrm{det}(\binom{0}{\id} - I^t_Q)\neq 0$.} in \cite[Cor.~3.13]{Suslin},~we have $\bbK_n(k;\bbZ/l^\nu)\simeq \bbZ/l^\nu$ if $n \geq 0$ is even and $\bbK_n(k;\bbZ/l^\nu)=0$ otherwise. Consequently, making use of the long exact sequence of algebraic $K$-theory groups associated to the preceding triangle of spectra, we obtain the following result:
\begin{corollary}\label{cor:main-3}
We have the following computation
$$ \bbK_n(L_Q;\bbZ/l^\nu)\simeq \left\{ \begin{array}{lll}
\mathrm{cokernel}\,\,\mathrm{of}\,\, \mathrm{M} & \mathrm{if} & n\geq 0\,\,\mathrm{even} \\
\mathrm{kernel}\,\,\mathrm{of}\,\,\mathrm{M} & \mathrm{if} & n\geq 0 \,\,\mathrm{odd} \\
0 & \mathrm{if} & n<0 \,,
\end{array} \right.
$$
where $\mathrm{M}$ stands for the homomorphism $\bigoplus^{v-v'}_{i=1} \bbZ/l^\nu \stackrel{\binom{0}{\id} - I^t_Q}{\too} \bigoplus^v_{i=1} \bbZ/l^\nu$.
\end{corollary}
Corollary \ref{cor:main-3} provides a complete and explicit computation of the mod-$l^\nu$ (nonconnective) algebraic $K$-theory of Leavitt path algebras. To the best of the author's knowledge, these computations are new in the literature. In particular, they yield a complete answer to the ``mod-$l^\nu$ version'' of Question $2$ raised by Gabe-Ruiz-Tomforde-Whalen in \cite[page~38]{LeavittK}. These computations lead also naturally to the following vanishing and divisibility properties of algebraic $K$-theory:
\begin{proposition}\label{prop:last}
\begin{itemize}
\item[(i)] If there exists a prime power $l^\nu$ and an even (resp. odd) integer $n'\geq 0$ such that $\bbK_{n'}(L_Q;\bbZ/l^\nu)\neq 0$, then for every even (resp. odd) integer $n\geq 0$ at least one of the groups $\bbK_n(L_Q), \bbK_{n-1}(L_Q)$ is non-zero.
\item[(ii)] If there exists a prime power $l^\nu$ such that $\bbK_n(L_Q; \bbZ/l^\nu)=0$ for every $n \geq 0$, then the groups $\bbK_n(L_Q), n \geq 0$, are uniquely $l^\nu$-divisible, \ie $\bbZ[1/l^\nu]$-modules.
\end{itemize}
\end{proposition}
\begin{proof}
Combine the universal coefficients sequence (see \cite[\S2.2.2]{book})
$$ 0 \too \bbK_n(L_Q) \otimes_\bbZ \bbZ/l^\nu \too \bbK_n(L_Q;\bbZ/l^\nu) \too {}_{l^\nu} \bbK_{n-1}(L_Q)\too 0 $$
with the computation of Corollary \ref{cor:main-3}.
\end{proof}
\begin{example}[Quivers without sinks]
Let $Q$ be a quiver without sinks. In this case, $\binom{0}{\id} - I^t_Q$ is a square matrix. If $l$ is a prime such that $l\nmid \mathrm{det}(\binom{0}{\id} - I^t_Q)$, then the homomorphism $\mathrm{M}$ of Corollary \ref{cor:main-3} is invertible. Consequently, $\bbK_n(L_Q;\bbZ/l^\nu)=0$ for every $n \geq 0$. Making use of Proposition \ref{prop:last}(ii), we then conclude that the algebraic $K$-theory groups $\bbK_n(L_Q), n \geq 0$, are uniquely $l^\nu$-divisible.
\end{example}
\subsection*{Schemes and stacks}
Let $X$ be a quasi-compact quasi-separated $k$-scheme. By applying the results/examples/remarks of \S\ref{sec:results} to the dg category $\cA=\perf_\dg(X)$, we obtain corresponding results/examples/remarks concerning the scheme $X$. For instance, Remark \ref{rk:very-last} yields an isomorphism between $KH(\perf_\dg(X)\otimes L_{l^\nu})$ and $KH(X;\bbZ/l^\nu)$. When $l \nmid \mathrm{char}(k)$, the latter spectrum is isomorphic to $\bbK(X;\bbZ/l^\nu)$. Roughly speaking, the dg category $\perf_\dg(X) \otimes L_{l^\nu}$ may be understood as the ``noncommutative mod-$l^\nu$ Moore object of $X$''. More generally, we can consider the dg category $\cA=\perf_\dg(\cX)$ of perfect complexes of an algebraic stack $\cX$. In the particular case of a quotient stack $\cX=[X/G]$, with $G$ an algebraic group scheme acting on $X$, Remark \ref{rk:very-last} yields an isomorphism between $KH(\perf_\dg([X/G])\otimes L_{l^\nu})$ and the mod-$l^\nu$ $G$-equivariant homotopy $K$-theory spectrum $KH^G(X;\bbZ/l^\nu)$. When $l \nmid \mathrm{char}(k)$, the latter spectrum is isomorphic to $\bbK^G(X;\bbZ/l^\nu)$.
\section{Preliminaries}
\subsection{Dg categories}\label{sub:dg}
Let $(\cC(k),\otimes, k)$ be the category of cochain complexes of $k$-vector spaces. A {\em dg category $\cA$} is a category enriched over $\cC(k)$ and a {\em dg functor $F\colon \cA \to \cB$} is a functor enriched over $\cC(k)$; consult Keller's ICM survey \cite{ICM-Keller}. 

Let $\cA$ be a dg category. The opposite dg category $\cA^\op$ has the same objects as $\cA$ and $\cA^\op(x,y):=\cA(y,x)$. A {\em right dg $\cA$-module} is a dg functor $\cA^\op \to \cC_\dg(k)$ with values in the dg category $\cC_\dg(k)$ of complexes of $k$-vector spaces. Let us denote by $\cC(\cA)$ the category of right dg $\cA$-modules. Following \cite[\S3.2]{ICM-Keller}, the {\em derived category $\cD(\cA)$ of $\cA$} is defined as the localization of $\cC(\cA)$ with respect to the objectwise quasi-isomorphisms. We write $\cD_c(\cA)$ for the subcategory of compact objects. 

A dg functor $F\colon\cA\to \cB$ is called a {\em derived Morita equivalence} if it induces an equivalence of categories $\cD(\cA) \simeq \cD(\cB)$; see \cite[\S4.6]{ICM-Keller}. As proved in \cite[Thm.~1.37]{book}, $\dgcat(k)$ admits a Quillen model structure whose weak equivalences are the derived Morita equivalences. Let $\Hmo(k)$ be the associated homotopy category.

The {\em tensor product $\cA\otimes\cB$} of (small) dg categories is defined as follows: the set of objects is the cartesian product and $(\cA\otimes\cB)((x,w),(y,z)):= \cA(x,y) \otimes \cB(w,z)$. As explained in \cite[\S2.3]{ICM-Keller}, this construction gives rise to a symmetric monoidal structure on $\dgcat(k)$, which descends to the homotopy category $\Hmo(k)$. 

A {\em dg $\cA\text{-}\cB$-bimodule $\mathrm{B}$} is a dg functor $\cA\otimes \cB^\op \to \cC_\dg(k)$ or equivalently a right dg $(\cA^\op \otimes \cB)$-module. Associated to a dg functor $F\colon \cA\to \cB$, we have the dg $\cA\text{-}\cB$-bimodule ${}_F\cB\colon \cA\otimes \cB^\op \to \cC_\dg(k), (x,z) \mapsto \cB(z, F(x))$. Let us write $\rep(\cA,\cB)$ for the full triangulated subcategory of $\cD(\cA^\op \otimes \cB)$ consisting of those dg $\cA\text{-}\cB$-bimodules $\mathrm{B}$ such that for every object $x \in \cA$ the associated right dg $\cB$-module $\mathrm{B}(x,-)$ belongs to $\cD_c(\cB)$. Clearly, the dg $\cA\text{-}\cB$-bimodules ${}_F\cB$ belongs to $\rep(\cA,\cB)$.

As explained in \cite[\S 1.6.3]{book}, there is a natural bijection between $\Hom_{\Hmo(k)}(\cA,\cB)$ and the set of isomorphism classes of $\rep(\cA,\cB)$. Under this bijection, the composition law corresponds to the tensor product of dg bimodules. 
\begin{notation}\label{not:unital}
Given a non-unital homomorphism $\phi\colon A \to B$ between unital $k$-algebras, let us denote by ${}_\phi B$ the $A\text{-}B$-bimodule $\phi(1)B$ equipped with the $A\text{-}B$-action $a\cdot \phi(1)B\cdot b:=\phi(a)Bb$. Note that ${}_\phi B$ belongs to $\rep(A,B)$.
\end{notation}

\subsection*{Square-zero extensions}
Let $\cA$ be a dg category and $\mathrm{B}$ a dg $\cA\text{-}\cA$-bimodule. The {\em square-zero extension $\cA\ltimes\! \mathrm{B}$ of $\cA$ by $\mathrm{B}$} is the dg category with the same objects as $\cA$ and complexes of $k$-vector spaces $(\cA\ltimes\! \mathrm{B})(x,y):=\cA(x,y)\oplus \mathrm{B}(y,x)$. Given morphisms $(f,f')\in (\cA\ltimes\! \mathrm{B})(x,y)$ and $(g,g') \in (\cA\ltimes\! \mathrm{B})(y,z)$, the composition $(g,g')\circ (f,f')$ is defined as $(g\circ f, g' f + g f')$.
\subsection*{Dg orbit categories}
Let $F\colon \cA\to \cA$ be an equivalence of dg categories. Following Keller \cite[\S5.1]{orbit}, the associated {\em dg orbit category $\cA/F^\bbZ$} has the same objects as $\cA$ and complexes of $k$-vector spaces $(\cA/F^\bbZ)(x,y):= \bigoplus_{n \in \bbZ} \cA(x,F^n(y))$. Given objects $x, y, z$ and morphisms $\mathrm{f}=\{f_n\}_{n \in \bbZ} \in \bigoplus_{n \in \bbZ} \cA(x,F^n(y))$ and $\mathrm{g}=\{g_n\}_{n \in \bbZ} \in \bigoplus_{n \in \bbZ} \cA(y,F^n(z))$, the $m^{\mathrm{th}}$-component of $\mathrm{g}\circ \mathrm{f}$ is defined as $\sum_n F^n(g_{m-n})\circ f_n$. When $\cA$ is a $k$-algebra $A$, the dg functor $F$ reduces to an isomorphism $\phi\colon A\stackrel{\sim}{\to} A$ and the dg orbit category $\cA/F^\bbZ$ to the skew Laurent polynomial algebra $A \rtimes_\phi \bbZ$.

Let us write $\cA/F^\bbN$ for the dg category with the same objects as $\cA$ and complexes of $k$-vector spaces $(\cA/F^\bbN)(x,y):=\bigoplus_{n \geq 0} \cA(x,F^n(y))$. The composition law is defined as above. By construction, we have a canonical dg functor $\cA/F^\bbN \to \cA/F^\bbZ$.

\section{Dg categories of idempotents}
Let $A$ be a (not necessarily unital) $k$-algebra.
\begin{definition}
The {\em dg category of idempotents of $A$}, denoted by $\underline{A}$, is defined as follows: the objects are the symbols ${\bf e}$ with $e$ an idempotent of $A$; the (complexes of) $k$-vector spaces $\underline{A}({\bf e},{\bf e'})$ are given by $eAe'$; the composition law is induced by the multiplication in $A$; the identity of the object ${\bf e}$ is the idempotent $e$. 
\end{definition}
\begin{notation}
Let $\mathrm{alg}(k)$ be the category of (not necessarily unital) $k$-algebras and (not necessarily unital) $k$-algebra homomorphisms. 
\end{notation}
Note that the preceding construction gives rise to the following functor:
\begin{eqnarray}\label{eq:functor-Alg}
\mathrm{alg}(k) \too \dgcat(k) & A\mapsto \underline{A} & \phi \mapsto \underline{\phi}\,.
\end{eqnarray} 
\begin{lemma}\label{lem:aux1}
The functor \eqref{eq:functor-Alg} preserves filtered colimits.
\end{lemma}
\begin{proof}
Consider a filtered diagram $\{A_i\}_{i \in I}$ in $\mathrm{alg}(k)$ with colimit $A$. Given an idempotent element $e$ of $A$, there exists an index $i' \in I$ and an idempotent $e_{i'} \in A_{i'}$ such that $e$ is the image of $e_{i'}$ under $A_{i'} \to A$. This implies that the induced dg functor $\mathrm{colim}_i \underline{A_i} \to \underline{A}$ is not only (essentially) surjective but also fully-faithful. 
\end{proof}
Given a unital $k$-algebra $A$ with unit $1$, let us write $\iota\colon A \to \underline{A}$ for the (unique) dg functor sending the single object of $A$ to the symbol ${\bf 1}$.
\begin{lemma}\label{lem:aux-new}
The dg functor $\iota$ is a derived Morita equivalence.
\end{lemma}
\begin{proof}
Note first that the dg functor $\iota$ is fully-faithful. Given an idempotent element $e$ of $A$, the morphisms ${\bf 1} \stackrel{e}{\to} {\bf e}$ and ${\bf e} \stackrel{e}{\to} {\bf 1}$ present the object ${\bf e}$ as a direct summand of ${\bf 1}$. This allows us to conclude that $\iota$ is a derived Morita equivalence.
\end{proof}
\begin{remark}\label{rk:comparison-1}
Given a non-unital homomorphism $\phi\colon A \to B$ between unital $k$-algebras, note that ${}_\iota \underline{B} \circ {}_\phi B = {}_{\underline{\phi}} \underline{B} \circ {}_\iota \underline{A}$ in the homotopy category $\Hmo(k)$.
\end{remark}
Let $A$ be a unital $k$-algebra and $M_2(A)$ the associated $k$-algebra of $2\times 2$ matrices. Consider the following non-unital homomorphisms 
\begin{eqnarray*}
j_{1}, j_{2}\colon A \too M_2(A)&  a \mapsto \begin{pmatrix} a & 0 \\ 0& 0\end{pmatrix} & a \mapsto \begin{pmatrix} 0 & 0 \\ 0& a\end{pmatrix}\,.
\end{eqnarray*}
Note that if there exist elements $t_+,t_-$ of $A$ such that $t_-t_+=1$, then we can also consider the non-unital homomorphism $\phi^\pm\colon A \to A, a \mapsto t_+ a t_-$.
\begin{proposition}\label{prop:aux3}
\begin{itemize}
\item[(i)] The dg functors $\underline{j_{1}}$ and $\underline{j_{2}}$ are derived Morita equivalences. Moreover, their images in the homotopy category $\Hmo(k)$ are the same.
\item[(ii)] The dg functor $\underline{\phi^\pm}$ is a derived Morita equivalence. Moreover, its image in the homotopy category $\Hmo(k)$ is the identity morphism.
\end{itemize}
\end{proposition}
\begin{proof}
(i) Recall first that a dg functor $F\colon \cA\to \cB$ is a derived Morita equivalence if and only if its image ${}_F\cB$ in the homotopy category $\Hmo(k)$ is invertible. Thanks to Lemma \ref{lem:aux-new} and Remark \ref{rk:comparison-1}, it suffices then to show that the $A\text{-}M_2(A)$-bimodules ${}_{j_{1}} M_2(A)$ and ${}_{j_{2}}M_2(A)$ are invertible in $\Hmo(k)$. Note that their inverses are given by the $M_2(A)\text{-}A$-bimodules $M_2(A)j_{1}(1)$ and $M_2(A)j_{2}(1)$, respectively. This shows the first claim. The second claim follows from the isomorphism ${}_{j_{1}} M_2(A) \stackrel{\sim}{\to} {}_{j_{2}} M_2(A)$ of $A\text{-}M_2(A)$-bimodules given by $\begin{pmatrix} a& b \\ 0 & 0\end{pmatrix} \mapsto \begin{pmatrix} 0& 0 \\ a & b\end{pmatrix}$.

(ii) Consider the following non-unital homomorphism
\begin{eqnarray*}
\varphi^\pm\colon M_2(A) \too M_2(A) && \begin{pmatrix} a& b \\ c & d\end{pmatrix} \mapsto \begin{pmatrix} t_+& 0 \\ 0 & 1\end{pmatrix} \begin{pmatrix} a& b \\ c & d\end{pmatrix} \begin{pmatrix} t_-& 0 \\ 0 & 1\end{pmatrix}\,.
\end{eqnarray*}
Note that $j_1\circ \phi^\pm = \varphi^\pm \circ j_1$ and $\varphi^\pm \circ j_2 = j_2$ in the category $\mathrm{alg}(k)$. By applying the functor \eqref{eq:functor-Alg}, we hence conclude from item (i) that the dg functor $\underline{\phi^\pm}$ is not only a derived Morita equivalence but moreover that its image in the homotopy category $\Hmo(k)$ is the identity morphism.
\end{proof}
\section{Proof of Theorem \ref{thm:main1}}\label{sec:proof1}
Consider the sequential colimit diagram $A\stackrel{\phi}{\to} A \stackrel{\phi}{\to} \cdots \to C$ in the category $\mathrm{alg}(k)$. Note that the $k$-algebra $C$ is non-unital and that the homomorphism $\phi$ gives rise to an isomorphism $\hat{\phi}\colon C \stackrel{\sim}{\to}C$. Let us denote by $C \rtimes_{\hat{\phi}}\bbN$, resp. $C \rtimes_{\hat{\phi}}\bbZ$, the associated skew polynomial algebra, resp. skew Laurent polynomial algebra. Note also that $\phi$ extends to (non-unital) homomorphisms $\phi\colon A[t_+,;\phi] \to A[t_+;\phi], a \mapsto t_+ a t_-$, and $\phi^\pm\colon A[t_+,t_-;\phi] \to A[t_+,t_-;\phi], a \mapsto t_+ a t_-$  Under these notations, we have the following sequential colimit diagrams:
\begin{equation}
A[t_+;\phi] \stackrel{\phi}{\too} A[t_+;\phi] \stackrel{\phi}{\too} \cdots \too C\rtimes_{\hat{\phi}}\bbN \label{eq:seq-1}
\end{equation}
\begin{equation}\label{eq:seq-2}
A[t_+,t_-;\phi] \stackrel{\phi^\pm}{\too} A[t_+,t_-;\phi] \stackrel{\phi^\pm}{\too} \cdots \too C\rtimes_{\hat{\phi}}\bbZ\,.
\end{equation}
\begin{lemma}\label{lem:new1}
The dg category of idempotents of $C\rtimes_{\hat{\phi}}\bbN$, resp. $C\rtimes_{\hat{\phi}}\bbZ$, is derived Morita equivalent to the dg category $\underline{C}/\underline{\hat{\phi}}^\bbN$, resp. $\underline{C}/\underline{\hat{\phi}}^\bbZ$.
\end{lemma}
\begin{proof}
We focus ourselves in the algebra $C \rtimes_{\hat{\phi}} \bbZ$ and in the dg orbit category~$\underline{C}/\underline{\hat{\phi}}^\bbZ$; the proof of the other case is similar. Let us denote by $1_n$ the unit of the $n^{\mathrm{th}}$ copy of $A[t_+,t_-;\phi]$ and by $e_n$ the image of $1_n$ under the induced homomorphism $A[t_+,t_-;\phi] \to C \rtimes_{\hat{\phi}}\bbZ$. Given an idempotent element $e$ of $C \rtimes_{\hat{\phi}}\bbZ$, there exists an integer $n \gg 0$ such that ${\bf e}$ is a direct summand of ${\bf e_n}$. Since $e_n$ belongs to $C\subset C\rtimes_{\hat{\phi}}\bbZ$, this allows us to conclude, in particular, that the dg category of idempotents of $C \rtimes_{\hat{\phi}}\bbZ$ is derived Morita equivalent to its full dg category of symbols ${\bf e}$ with $e$ an idempotent element of $C\subset C\rtimes_{\hat{\phi}}\bbZ$. Given any two such symbols ${\bf e}$ and ${\bf e'}$, note that we have the following equalities of (complexes of) $k$-vector spaces:
\begin{eqnarray*}
(\underline{C \rtimes_{\hat{\phi}}\bbZ})({\bf e}, {\bf e'}) & := & e(C \rtimes_{\hat{\phi}}\bbZ)e' \\
& = & \bigoplus_{n \in \bbZ} \,eC \hat{\phi}^n(e') \\
 & = &  \bigoplus _{n \in \bbZ} \,\underline{C}({\bf e}, \underline{\hat{\phi}}^n({\bf e'})) =: (\underline{C}/\underline{\hat{\phi}}^\bbZ)({\bf e},{\bf e'})\,.
\end{eqnarray*}
Under these equalities, the composition law of the dg category of idempotents of $C \rtimes_{\hat{\phi}}\bbZ$ corresponds to the composition law of the dg orbit category $\underline{C}/\underline{\hat{\phi}}^\bbZ$. 
\end{proof}
By combining Lemmas \ref{lem:aux1}-\ref{lem:aux-new} with Remark \ref{rk:comparison-1}, we conclude from Lemma \ref{lem:new1} that \eqref{eq:seq-1}-\eqref{eq:seq-2} give rise to the following sequential (homotopy) colimit diagrams
\begin{equation}\label{eq:seq-11}
A[t_+;\phi] \stackrel{{}_{\phi}A[t_+;\phi]}{\too}  A[t_+;\phi] \stackrel{{}_{\phi}A[t_+;\phi]}{\too} \cdots \too \underline{C}/\underline{\hat{\phi}}^\bbN
\end{equation}
\begin{equation}\label{eq:seq-22}
A[t_+,t_-;\phi] \stackrel{{}_{\phi^\pm}A[t_+,t_-;\phi]}{\too} A[t_+,t_-;\phi] \stackrel{{}_{\phi^\pm}A[t_+,t_-;\phi]}{\too} \cdots \too \underline{C}/\underline{\hat{\phi}}^\bbZ
\end{equation}
in the homotopy category $\Hmo(k)$. Proposition \ref{prop:aux3}(ii) (with $A[t_+,t_-;\phi]$ instead of $A$) leads then automatically to the following result:
\begin{proposition}\label{prop:new}
The (transfinite) composition \eqref{eq:seq-22} is an isomorphism. Consequently, the dg categories $A[t_+,t_-;\phi]$ and $\underline{C}/\underline{\hat{\phi}}^\bbZ$ are derived Morita equivalent.
\end{proposition}
Now, consider the square-zero extension $\underline{C}[\epsilon]:=\underline{C}\ltimes\! ({}_{\underline{\hat{\phi}}}\underline{C}[1])$ of $\underline{C}$ by the suspension ${}_{\underline{\hat{\phi}}} \underline{C} [1]$ of the dg $\underline{C}\text{-}\underline{C}$-bimodule ${}_{\underline{\hat{\phi}}} \underline{C}$ associated to $\underline{\hat{\phi}}$. Consider also the dg $\underline{C}[\epsilon]\text{-}(\underline{C}/\underline{\hat{\phi}}^\bbN)$-bimodule $\hat{\mathrm{B}}$ introduced in \cite[\S4]{orbit}; denoted by $\mathrm{B}'$ in {\em loc. cit.}
\begin{lemma}\label{lem:new2}
We have the following sequential (homotopy) colimit diagram
$$ A[\epsilon] \stackrel{{}_\phi A[\epsilon]}{\too} A[\epsilon] \stackrel{{}_\phi A[\epsilon]}{\too} \cdots \too \underline{C}[\epsilon]$$
in the homotopy category $\Hmo(k)$.
\end{lemma}
\begin{proof}
Thanks to Lemma \ref{lem:aux1}, we have $\underline{A} \stackrel{\underline{\phi}}{\to} \underline{A} \stackrel{\underline{\phi}} \to \cdots \to \underline{C}$
in the category $\dgcat(k)$. Consider the square-zero extension $\underline{A}[\epsilon]:=\underline{A}\ltimes\! ({}_{\underline{\phi}}\underline{A}[1])$ of $\underline{A}$ by the suspension ${}_{\underline{\phi}} \underline{A} [1]$ of the dg $\underline{A}\text{-}\underline{A}$-bimodule ${}_{\underline{\phi}} \underline{A}$ associated to $\underline{\phi}$. Similarly to the proof of Lemma \ref{lem:aux1}, we have the following induced sequential colimit diagram 
$$\underline{A}[\epsilon] \stackrel{\underline{\phi}}{\too} \underline{A}[\epsilon] \stackrel{\underline{\phi}}{\too} \cdots \too \underline{C}[\epsilon]\,.$$ 
Note that the dg functor $A[\epsilon] \to \underline{A}[\epsilon]$ sending the single object $A[\epsilon]$ to the symbol ${\bf 1}$, where $1$ is the unit of $A$, is a derived Morita equivalence. Under such derived Morita equivalence the dg functor $\underline{\phi}\colon \underline{A}[\epsilon] \to \underline{A}[\epsilon]$ corresponds to the morphism ${}_\phi A[\epsilon]\colon A[\epsilon] \to A[\epsilon]$ in the homotopy category $\Hmo(k)$. This concludes the proof.
\end{proof}
By combining Lemma \ref{lem:new2} with the sequential (homotopy) colimit diagram \eqref{eq:seq-11}, we obtain the following sequential (homotopy) colimit diagram
\begin{equation}\label{eq:diagram-big}
\xymatrix@C=2.5em@R=1.5em{
\underline{C}[\epsilon] \ar[rr]^-{\hat{\mathrm{B}}} && \underline{C}/\underline{\hat{\phi}}^\bbN \\
\vdots \ar[u] && \vdots \ar[u] \\
A[\epsilon] \ar[u]^-{{}_\phi A[\epsilon]} \ar[rr]^-{\mathrm{B}} && A[t_+;\phi] \ar[u]_-{{}_\phi A[t_+;\phi]} \\
A[\epsilon] \ar[u]^-{{}_\phi A[\epsilon]}  \ar[rr]^-{\mathrm{B}} && A[t_+;\phi] \ar[u]_-{{}_\phi A[t_+;\phi]}
}
\end{equation}
in the homotopy category $\Hmo(k)$. As explained in \cite[\S4]{orbit}, given any localizing invariant $E$, we have a distinguished triangle
\begin{equation}\label{eq:triangle-key}
E(\underline{C}[\epsilon])\stackrel{E(\hat{\mathrm{B}})}{\too} E(\underline{C}/\underline{\hat{\phi}}^\bbN) \too E(\underline{C}/\underline{\hat{\phi}}^\bbZ) \stackrel{\partial}{\too} \Sigma E(\underline{C}[\epsilon])\,.
\end{equation}
Therefore, by combining the diagram \eqref{eq:diagram-big} with Proposition \ref{prop:new}, we obtain the searched homotopy colimit diagram \eqref{eq:sequence}. This concludes the proof of the~first~claim.

We now prove the second claim. Let $E$ be an $\bbA^1$-homotopy invariant. As explained in \cite[Prop.~4.6]{orbit}, $E(\mathrm{B})$ reduces to $\id -E({}_\phi A)\colon E(A) \to E(A)$. Similarly, the morphism $E(\hat{\mathrm{B}})$ reduces to $\id- E({}_{\underline{\hat{\phi}}} \underline{C})\colon E(\underline{C}) \to E(\underline{C})$. Therefore, making use of Lemma \ref{lem:aux-new} and of Remark \ref{rk:comparison-1}, we observe that by applying the functor $E$ to \eqref{eq:diagram-big} we obtain (up to isomorphism) the following sequential colimit diagram:
\begin{equation}\label{eq:diagram-big2}
\xymatrix@C=2.5em@R=1.5em{
E(\underline{C})\ar[rr]^-{\id - E(\underline{\hat{\phi}})} && E(\underline{C}) \\
\vdots \ar[u] && \vdots \ar[u] \\
E(\underline{A}) \ar[u]^-{E(\underline{\phi})} \ar[rr]^-{\id - E(\underline{\phi})} && E(\underline{A})  \ar[u]_-{E(\underline{\phi})} \\
E(\underline{A})  \ar[u]^-{E(\underline{\phi})}  \ar[rr]^-{\id - E(\underline{\phi})} && E(\underline{A})  \ar[u]_-{E(\underline{\phi})}\,.
}
\end{equation}
We now claim that the induced (transfinite) composition
\begin{equation}\label{eq:hocofib}
\mathrm{hocofib}(\id - E(\underline{\phi})) \stackrel{E(\underline{\phi})}{\too} \mathrm{hocofib}(\id - E(\underline{\phi})) \stackrel{E(\underline{\phi})}{\too} \cdots \too \mathrm{hocofib}(\id -E(\underline{\hat{\phi}})) \end{equation}
is an isomorphism.
As explained in \cite[Thm.~8.25]{book}, the functor $U_\loc^{\bbA^1}$ is the initial $\bbA^1$-homotopy invariant. Therefore, it suffices to prove the latter claim in the particular case where $E=U^{\bbA^1}_\loc$. By construction, we have a factorization 
$$ U_\loc^{\bbA^1}\colon \dgcat(k) \stackrel{U_\add}{\too} \Mot_\add(k) \stackrel{\gamma}{\too} \Mot_\loc^{\bbA^1}(k)\,,$$
where $\Mot_\add(k)$ is a certain compactly generated triangulated category of noncommutative mixed motives, $U_\add$ is a certain functor sending sequential (homotopy) colimits to sequential homotopy colimits, and $\gamma$ is a certain homotopy colimit preserving functor; consult \cite[\S 8.4.2]{book} for details. The triangulated category $\Mot_\add(k)$ is moreover enriched over spectra; we write $\Hom_\Spt(-,-)$ for this enrichment. Let $N\!\!M$ be a compact object of $\Mot_\add(k)$. In order to prove our claim, it is then enough to show that the (transfinite) composition obtained by applying the functor $\Hom_\Spt(N\!\!M,-)$ to \eqref{eq:hocofib} (with $E=U_\add$) is an isomorphism.
%
%
Since the spectrum $\Hom_\Spt(N\!\!M, U_\add(\underline{C}))$ is the sequential homotopy colimit of $\Hom_\Spt(N\!\!M, U_\add(\underline{A}))$, with respect to the transition morphism(s) $\Hom_\Spt(N\!\!M, U_\add(\underline{\phi}))$, the proof follows now automatically from the general result \cite[Lem.~3.3]{Munster} concerning spectra. This finishes the proof of Theorem \ref{thm:main1}.
\section{Proof of the generalization of Theorem \ref{thm:main1}}\label{sub:generalization}
The triangulated category $\Mot_\loc(k)$ carries a symmetric monoidal structure making the functor $U_\loc$ symmetric monoidal; see \cite[\S8.3.1]{book}. Therefore, the distinguished triangle \eqref{eq:triangle-key} (with $E=U_\loc$) gives rise to the distinguished triangle:
$$
U_\loc(\cA\otimes \underline{C}[\epsilon]) \stackrel{U_\loc({}_{\id}\cA\otimes \hat{\mathrm{B}})}{\too} U_\loc(\cA\otimes \underline{C}/\underline{\hat{\phi}}^\bbN) \too U_\loc(\cA\otimes \underline{C}/\underline{\hat{\phi}}^\bbZ) \stackrel{\partial}{\too} \Sigma U_\loc(\cA\otimes \underline{C}[\epsilon])\,.
$$
Since the functor $\cA\otimes -$ preserves (sequential) homotopy colimits, the combination of the preceding triangle with the commutative diagram \eqref{eq:diagram-big} and with Proposition \ref{prop:new} leads then to the following sequential homotopy colimit diagram
$$ \mathrm{hocofib}\,U_\loc({}_{\id} \cA \otimes \mathrm{B}) \too \mathrm{hocofib}\,U_\loc({}_{\id} \cA \otimes \mathrm{B}) \too \cdots \too U_\loc(\cA\otimes A[t_+,t_-;\phi])\,,$$
where the transition morphism(s) is induced by the corner isomorphism $\phi$. The proof of the first claim follows now automatically from the fact that $U_\loc$ is the initial localizing invariant; see \cite[Thm.~8.5]{book}.

The triangulated category $\Mot^{\bbA^1}_\loc(k)$ carries a symmetric monoidal structure making the functor $U_\loc^{\bbA^1}$ symmetric monoidal; see \cite[\S8.5.2]{book}. Therefore, the distinguished triangle \eqref{eq:triangle-searched} (with $E=U_\loc^{\bbA^1}$) gives rise to the distinguished triangle:
$$ U_\loc^{\bbA^1}(\cA\otimes A) \stackrel{\id - U_\loc^{\bbA^1}({}_{\id}\cA \otimes_{}\phi A)}{\too} U_\loc^{\bbA^1}(\cA\otimes A) \too U_\loc^{\bbA^1}(\cA\otimes A[t_+,t_-;\phi]) \stackrel{\partial}{\too} \Sigma U_\loc^{\bbA^1}(\cA\otimes A)\,.$$
The proof of the second claim follows now automatically from the fact that $U^{\bbA^1}_\loc$ is the initial $\bbA^1$-homotopy invariant; see \cite[Thm.~8.25]{book}.
\section{Proof of Theorem \ref{thm:main2}}\label{sec:proof2}
Similarly to the arguments used in \S\ref{sub:generalization}, it suffices to prove Theorem \ref{thm:main2} in the particular case where $E=U_\loc^{\bbA^1}$ and $\cA=k$. As mentioned in Example \ref{ex:Leavitt-path}, the Leavitt path algebra $L:=L_Q$ is a corner skew Laurent polynomial algebra. Let $L_0$ be the homogeneous component of degree $0$ and $\phi\colon L_0 \stackrel{\sim}{\to} eL_0e$ the ``corner'' isomorphism. Thanks to Theorem \ref{thm:main1} (with $E=U_\loc^{\bbA^1}$), we have a triangle
\begin{equation}\label{eq:triangle-last}
U_\loc^{\bbA^1}(L_0) \stackrel{\id -U_\loc^{\bbA^1}({}_\phi L_0)}{\too} U_\loc^{\bbA^1}(L_0) \too U_\loc^{\bbA^1}(L_Q) \stackrel{\partial}{\too} \Sigma U_\loc^{\bbA^1}(L_0)
\end{equation}
in the category $\Mot_\loc^{\bbA^1}(k)$. Following Ara-Brustenga-Corti\~nas \cite[\S5]{Munster}, the $k$-algebra $L_0$ admits a ``length'' filtration $L_0=\bigcup_{n=0}^\infty L_{0,n}$. Concretely, $L_{0,n}$ is the $k$-linear span of the elements of the form $\sigma \varsigma^\ast$, where $\sigma$ and $\varsigma$ are paths such that $r(\sigma)=r(\varsigma)$ and $\mathrm{deg}(\sigma)=\mathrm{deg}(\varsigma)=n$. It turns out that the $k$-algebra $L_{0,n}$ is isomorphic to the product of $(n+1)v' + (v-v')$ matrix algebras with $k$-coefficients. Making use of the (derived) Morita equivariance between a matrix algebra with $k$-coefficients and $k$, we hence conclude that $U_\loc^{\bbA^1}(L_{0,n})$ is isomorphic to the direct sum of $(n+1)v' + (v-v')$ copies of $U_\loc^{\bbA^1}(k)$. Recall from \cite[Thm.~8.28]{book} that we we have an isomorphism $\Hom_{\Mot_\loc^{\bbA^1}(k)}(U_\loc^{\bbA^1}(k),U_\loc^{\bbA^1}(k))\simeq K_0(k)\simeq \bbZ$. Under this identification, the inclusion $L_{0,n} \subset L_{0, n+1}$ corresponds to the matrix morphism~(see \cite[\S5]{Munster}):
\begin{equation}\label{eq:matrix1}
\begin{pmatrix} \id & 0 \\ 0 & I_Q^t\end{pmatrix} \colon \bigoplus_{i=1}^{(n+1)v' + (v-v')} U_\loc^{\bbA^1}(k) \too \bigoplus_{i=1}^{(n+1)v' + v} U_\loc^{\bbA^1}(k)\,.
\end{equation}
In the same vein, the homomorphism $\phi\colon L_{0,n} \to L_{0,n+1}$, which increases the degree of the filtration by $1$, corresponds to the matrix morphism 
\begin{equation}\label{eq:matrix2}
\binom{0}{\id} \colon \bigoplus_{i=1}^{nv'+v} U_\loc^{\bbA^1}(k) \too \bigoplus_{i=1}^{(n+1)v' + v} U_\loc^{\bbA^1}(k)\,.
\end{equation}
Since the functor $U_\loc^{\bbA^1}$ sends sequential (homotopy) colimits to sequential homotopy colimits, we hence obtain the following sequential homotopy colimit diagram
\begin{equation}\label{eq:diagram-long}
\xymatrix{
U_\loc^{\bbA^1}(L_{0,0}) \ar[r]^-{\eqref{eq:matrix1}} \ar[d]_-{\eqref{eq:matrix1}\text{-}\eqref{eq:matrix2}} & U_\loc^{\bbA^1}(L_{0,1}) \ar[d]_-{\eqref{eq:matrix1}\text{-}\eqref{eq:matrix2}}  \ar[r]^-{\eqref{eq:matrix1}} & \cdots \ar[r] & U_\loc^{\bbA^1}(L_0) \ar[d]^-{\id- U_\loc^{\bbA^1}({}_\phi L_0)} \\
U_\loc^{\bbA^1}(L_{0,1}) \ar[r]_-{\eqref{eq:matrix1}} & U_\loc^{\bbA^1}(L_{0,2}) \ar[r]_-{\eqref{eq:matrix1}} & \cdots \ar[r] & U_\loc^{\bbA^1}(L_0)\,.
}
\end{equation}
Simple matrix manipulations show that the homotopy cofibers of the vertical morphisms of the diagram \eqref{eq:diagram-long} are all equal to the homotopy cofiber of the morphism $\binom{0}{\id} - I_Q^t\colon \bigoplus^{v-v'}_{i=1} U_\loc^{\bbA^1}(k)\to \bigoplus^v_{i=1} U_\loc^{\bbA^1}(k)$. This allows us then to conclude that distinguished triangle \eqref{eq:triangle-last} yields the following distinguished triangle
$$ \oplus^{v-v'}_{i=1} U_\loc^{\bbA^1}(k) \stackrel{\binom{0}{\id} - I_Q^t}{\too} \oplus^v_{i=1} U_\loc^{\bbA^1}(k) \too U_\loc^{\bbA^1}(L_Q) \stackrel{\partial}{\too} \oplus^{v-v'}_{i=1} \Sigma U_\loc^{\bbA^1}(k)\,.$$
Consequently, the proof is finished.

\section{Proof of Proposition \ref{prop:very-last}}
By construction, the triangulated category $\Mot^{\bbA^1}_\loc(k)$ comes equipped with an action of the homotopy category of spectra (see \cite[\S A.3]{book}):
\begin{eqnarray*}
\Spt \times \Mot^{\bbA^1}_\loc(k) \too \Mot^{\bbA^1}_\loc(k) && (S,N\!\!M) \mapsto S\otimes N\!\!M\,.
\end{eqnarray*}
Consider the distinguished triangle of spectra $\bbS \stackrel{n\cdot -}{\to} \bbS \to \bbS/n \to \Sigma \bbS$, where $\bbS$ stands for the sphere spectrum. Since $\bbS/n \otimes U^{\bbA^1}_\loc(\cA)$ identifies with the mod-$n$ Moore object of $U^{\bbA^1}_\loc(\cA)$ and we have an isomorphism $\bbS/n\simeq \bbS/l_1^{\nu_1} \oplus \cdots \oplus \bbS/l_r^{\nu^r}$ in $\Spt$, we then conclude from Example \ref{ex:Leavitt} (with $E=U^{\bbA^1}_\loc$) that
$$ U^{\bbA^1}_\loc(\cA\otimes L_n) \simeq U^{\bbA^1}_\loc(\cA\otimes L_{l_1^{\nu_1}})\oplus \cdots \oplus U^{\bbA^1}_\loc(\cA\otimes L_{l_r^{\nu_r}})\,.$$
The proof follows now automatically from the fact that $U_\loc^{\bbA^1}$ is the initial $\bbA^1$-homotopy invariant; see \cite[Thm.~8.25]{book}.

\medbreak\noindent\textbf{Acknowledgments:} The author is very grateful to Guillermo Corti\~nas for sharing his views on Leavitt path algebras and their algebraic $K$-theory, for useful comments on a previous version of this note, and for mentioning the reference \cite{AK}.

\end{document}

\end{proof}